\documentclass[11pt]{amsart}
\usepackage[margin=1.05in]{geometry}
\usepackage{amsmath,amssymb,amsthm,mathtools}
\usepackage{booktabs}
\usepackage{array}
\usepackage{tikz}
\usetikzlibrary{arrows.meta,positioning,shapes.geometric}
\usepackage[colorlinks=true,linkcolor=blue,citecolor=blue,urlcolor=blue]{hyperref}
\numberwithin{equation}{section}
\newtheorem{theorem}{Theorem}[section]
\newtheorem{lemma}[theorem]{Lemma}
\newtheorem{proposition}[theorem]{Proposition}
\newtheorem{conjecture}[theorem]{Conjecture}
\newtheorem{definition}[theorem]{Definition}

\newtheorem*{proposition*}{Proposition}
\theoremstyle{remark}
\newtheorem{remark}[theorem]{Remark}
\newtheorem{example}[theorem]{Example}
\DeclareMathOperator{\Sym}{Sym}

\newcommand{\C}{\mathbb C}
\newcommand{\Q}{\mathbb Q}
\newcommand{\Z}{\mathbb Z}

\DeclareMathOperator{\Pol}{Pol}
\DeclareMathOperator{\Pl}{Pl}
\newcommand{\Dder}{\mathfrak D}
\newcommand{\NN}{\operatorname{NN}}
\newcommand{\stir}[2]{\genfrac{[}{]}{0pt}{}{#1}{#2}}

\newcommand{\StirTwo}[2]{\genfrac{\{} {\}}{0pt}{}{#1}{#2}}

\title[Positivity in classical enumerative geometry:a synchronized AI workflow]{Positivity in classical enumerative geometry: A Case Study in Synchronized AI-Assisted Mathematics}
\author{Gergely B\'erczi} 
\address{Department of Mathematics, Aarhus University}
\email{gergely.berczi@math.au.dk}
\author{L\'aszl\'o M. Feh\'er}
\address{Eötvös University Budapest and Alfréd Rényi Institute of Mathematics}
\email{lfeher63@gmail.com}
\date{April 2026}

\begin{document}

\begin{abstract}
We study the symmetric polynomial
\[
\prod_{\alpha\in A_{n,d}}\bigl(1+\alpha_1 x_1+\cdots+\alpha_n x_n\bigr),
\qquad A_{n,d}:=\{\alpha\in\Z_{\ge0}^n:|\alpha|=d\},
\]
which is the total Chern class of $\Sym^d(\C^n)$, viewed as a torus representation whose Chern roots are the weights $\alpha_1 x_1+\cdots+\alpha_n x_n$ for $\alpha\in A_{n,d}$. Its homogeneous degree-$k$ part $c_k(n,d)$ is the $k$-th Chern class of $\Sym^d(\C^n)$.  These Chern classes, together with their coefficients in various symmetric function bases, play a central role in enumerative geometry. Despite their simple definition, general closed formulas for their coefficients are subtle, and many structural properties of these classes have remained poorly understood. 

In this paper we prove several conjectures concerning their structure, establish explicit formulas, and study log-concavity properties for both the Chern classes and their $K$-theoretic analogue. In rank two, passing to the Schur basis and expanding the Schur coefficients in the binomial basis of $d$, we uncover a new binomial log-concavity phenomenon and prove refined positivity results.

The paper demonstrates a novel methodology: we combine several AI systems with human mathematical insight in a coordinated workflow, deploying each tool according to its strengths in experimental discovery, conjecture formation, symbolic proof construction, and verification. To our knowledge, this is one of the first detailed case studies of orchestrating multiple AI tools to make substantial progress on a coherent mathematical research project.
\end{abstract}

\maketitle
%\tableofcontents

\section{Introduction}

Chern classes of symmetric powers of vector bundles form a classical meeting point of intersection theory, enumerative geometry, and representation theory.  If $E$ is a rank-$n$ vector bundle with Chern roots $x_1,\ldots,x_n$, the splitting principle identifies the total Chern class of $\Sym^dE$ with
\begin{equation}\label{eq:intro-total}
  c(\Sym^dE)=
  \prod_{\substack{\alpha=(\alpha_1,\ldots,\alpha_n)\in\Z_{\ge0}^n\\|\alpha|=d}}
  \left(1+\alpha_1x_1+\cdots+\alpha_nx_n\right).
\end{equation}
We write
\[
  c(\Sym^dE)=1+c_1(n,d)+c_2(n,d)+\cdots,
\]
where $c_k(n,d)$ is the homogeneous degree $k$ part of the product.  Thus $c_k(n,d)$ is a symmetric polynomial in $x_1,\ldots,x_n$, and it can be expanded in several natural bases of the ring of symmetric functions.  The goal of this paper is twofold: first, to obtain structure theorems,  closed and recursive formulas for these classes and for related $K$-theoretic expressions; second, to study positivity and log-concavity phenomena that become visible after choosing suitable bases.

\subsection{Motivation} 
We recall two directions of related work that motivate the present paper.  The first concerns the explicit calculation of Chern classes of functorial constructions on vector bundles.  The splitting principle gives a universal algorithm: the Chern roots of $E\otimes F$, $\Sym^dE$, $\bigwedge^dE$, and more general Schur functors are obtained by applying the corresponding weight operations to the Chern roots of the input bundles.  However, converting these root formulas into compact expressions in Chern classes or symmetric-function bases is a subtle problem.  Lascoux gave formulas for Chern classes of tensor products and second symmetric and exterior powers, and Laksov--Lascoux--Thorup developed determinantal formulas for Chern and Segre classes associated with these constructions~\cite{LascouxTensor,LLT}.  Manivel's formulas for Chern classes of tensor products provide another systematic framework~\cite{ManivelChern}.  In the special case of $\Sym^2E$, the Laksov--Lascoux--Thorup determinant gives an especially elegant Schur expansion; we recall this formula below.  More recently, Chern plethysm, as developed by Billey--Rhoades--Tewari, has placed such expressions in a broader representation-theoretic and symmetric-function context~\cite{BRT}.

The top Chern class in \eqref{eq:intro-total} also has a direct enumerative interpretation.  If $S$ is the tautological bundle on a Grassmannian, then
\[
  c_{\rm top}(\Sym^dS^\vee)
\]
is the Euler class of the vector bundle whose zero scheme is the Fano scheme of linear subspaces lying on a degree-$d$ hypersurface.  In the expected zero-dimensional case, its integral gives the degree of the corresponding Fano scheme.  Such degrees have been studied by a variety of methods, including residue formulas, localization, and Schubert calculus; see, for example, Manivel~\cite{manivel1999surleshypersurfaces}, Ciliberto--Zaidenberg~\cite{ciliberto_zaidenberg2020onfanoschemes}, Feh\'er--Juh\'asz~\cite{feher2025polynomiality}, and Hiep~\cite{HiepDegree}.  The corresponding monodromy questions for finite Fano problems were studied by Harris~\cite{Harris} and, in broad generality, by Hashimoto--Kadets~\cite{HK} and Sottile--Yahl~\cite{SY}.

The second direction concerns positivity.  Two bases for symmetric polynomials play especially important roles in geometry: the elementary symmetric basis, which is the Chern-class basis, and the Schur basis, which is the natural basis of Schubert calculus.  A basic theorem of Fulton--Lazarsfeld identifies Schur polynomials as the fundamental numerically positive polynomials in the Chern classes of ample vector bundles~\cite{FultonLazarsfeld}.  This is one reason Schur positivity is often a stronger and more geometric statement than positivity in an arbitrary monomial basis.  In singularity theory, Thom polynomials provide a parallel source of positivity questions.  Rim\'anyi \cite{RimanyiThom} conjectured positivity properties for Thom polynomials of Morin singularities when written in the relative Chern classes; B\'erczi's work on Morin singularities discusses this conjecture and related strengthened positivity phenomena~\cite{BercziMorin}.  On the other hand, Pragacz and Weber proved broad Schur-positivity results for Thom polynomials of stable singularities, using the Fulton--Lazarsfeld theory together with Kazarian's classifying-space approach~\cite{PragaczWeber}.  Pragacz and collaborators also showed that Schur expansions often reveal recurrences and structural features that are difficult to see in the Chern-monomial basis~\cite{PragaczThomSchur,LascouxPragaczA3}.

Positivity is frequently accompanied by unimodality or log-concavity phenomena.  These properties are central in algebraic combinatorics and algebraic geometry: classical examples include unimodality of Hilbert functions and $h$-vectors, log-concavity conjectures for matroids, and the Hodge-theoretic proofs of the Rota--Heron--Welsh conjecture and related Mason-type conjectures~\cite{StanleyLogConcavity,Huh,AHK,BrandenHuh}.  The modern theory of Lorentzian and Hodge-theoretic polynomials gives a conceptual explanation for many such inequalities.  In this paper, the same philosophy leads us from closed formulas for the Chern classes $c_k(n,d)$ to refined positivity questions for their Schur coefficients.  

\subsection{Results} The first part of the paper concerns polynomiality in the two parameters
$n$ and $d$.  We study the weak polynomiality conjecture and a
stronger binomial-form conjecture from \cite{feher2025polynomiality} for the elementary-basis coefficients from 
\[
  c_k(n,d)=\sum_{\lambda\vdash k} f_\lambda(n,d)e_\lambda,
  \qquad e_\lambda=e_{\lambda_1}\cdots e_{\lambda_\ell}.
\]
The main result in this direction is Theorem~\ref{thm:strong-binomiality},
which proves more than the conjectural binomial form: for each fixed $k$,
all coefficients $f_\lambda(n,d)$ lie in the polynomial ring
\[
  \Q[T_1(n,d),\ldots,T_k(n,d)],
  \qquad
  T_r(n,d):=\binom{d+n-1}{n+r-1}.
\]
Thus the $k$ distinguished binomial coefficients $T_1,\ldots,T_k$
already generate all elementary-basis coefficients of $c_k(n,d)$.  The
proof reduces the Chern-root power sums to factorial moments of weak
compositions, using Lemma~\ref{lem:factorial-moments}, and then applies
Newton identities.

We then record several consequences and variants of this binomiality
theorem.  Proposition~\ref{prop:c2} gives the first nontrivial closed
formula, namely $c_2(n,d)$, and Example~\ref{ex:small-k-cknd} gives the
first universal expressions in the variables $T_r$.  The
Laksov--Lascoux--Thorup determinant for $d=2$ is recalled in
Theorem~\ref{thm:d2det}.  We also prove a multiplicative, $K$-theoretic
analogue: the normalized product $\NN_{n,d}$ satisfies
\[
  [z^k\zeta^k]\NN_{n,d}=c_k(n,d)
\]
by Proposition~\ref{prop:NN-diagonal}, and Theorem~\ref{thm:NN-binomiality}
shows that all coefficients $[z^q\zeta^m]\NN_{n,d}$ have the same
binomial dependence on $n$ and $d$.  Finally, we show leading-term
consequences in the elementary basis and rank-two Euler-class formulas for
the top Chern class.

The second part of the paper studies a refined rank-two positivity
phenomenon.  Write
\[
  c_k(2,d)
  =
  \sum_{0\le j\le \lfloor k/2\rfloor}
  A_{k,j}(d)\,s_{(k-j,j)}(x_1,x_2),
\]
and expand each Schur coefficient in the binomial basis of $d$:
\[
  A_{k,j}(d)=\sum_r B_{k,j,r}\binom dr.
\]
Conjecture~\ref{conj:A} asserts binomial positivity,
$B_{k,j,r}\ge0$, and Conjecture~\ref{conj:B} asserts binomial
log-concavity in the index $r$.  We prove first the general implication
that binomial positivity together with binomial log-concavity implies
ordinary log-concavity of the value sequence; this is
Theorem~\ref{thm:binomial-lc-implies-lc}.  We then prove the initial
rank-two cases: Theorem~\ref{thm:rank-two-known-A} proves positivity for
$j=0,1$, and Theorem~\ref{thm:rank-two-known-B} proves log-concavity for
$j=0$, using fixed-point-free permutations with prescribed cycle defect.

The first cases not covered by these initial arguments are settled in
Section~\ref{sec:detective-proofs}.  Theorem~\ref{thm:A1-solved} proves
Conjecture~\ref{conj:A} for $j=2$, that is,
\[
  B_{k,2,r}\ge0
  \qquad\text{for all }k,r.
\]
Theorem~\ref{thm:B1-solved} proves Conjecture~\ref{conj:B} for $j=1$,
namely
\[
  B_{k,1,r}^2\ge B_{k,1,r-1}B_{k,1,r+1}
  \qquad\text{for all }k,r.
\]
Thus the paper proves binomial positivity for the first three rank-two
Schur modes $j=0,1,2$, and binomial log-concavity for the first two modes
$j=0,1$.  The general higher-mode forms of
Conjectures~\ref{conj:A} and~\ref{conj:B} remain open.

The final section moves beyond symmetric powers and formulates an analogous shifted-binomial log-concavity problem for Pl\"ucker coefficients of coincident root strata; unlike the rank-two Chern-class results, this Pl\"ucker extension is presented as a conjectural direction together with a proof strategy and a precise obstruction. This is a
classical enumerative problem, going back to Pl\"ucker's work in the
1830s. Let $\lambda=(2^{e_2},3^{e_3},\ldots,m^{e_m})$
be a partition of $n$ without parts equal to $1$.  Equivalently, after choosing
an ordering of its parts, we write $\lambda=(\lambda_1,\ldots,\lambda_\ell)$ where $\ell=\sum_{q\ge2}e_q$.  
We also use the partition
\[
  \tilde\lambda
  :=
  (\lambda_1-1,\ldots,\lambda_\ell-1)
  =
  (1^{e_2},2^{e_3},\ldots,(m-1)^{e_m}),
\]
The number $|\tilde\lambda|$ is the total tangency codimension imposed by the prescribed multiplicities.

Choose $n_0$ and an admissible integer $i$ with
$0\le i\le |\tilde\lambda|$ and $|\tilde\lambda|+i=2(n_0-2)$.
For $d\ge |\lambda|$, the Pl\"ucker number
$\Pl_{\lambda;i}(d)$ is the number of lines
$L\subset \mathbb P(\mathbb C^{n_0})$ satisfying the following conditions
with respect to a general degree $d$ hypersurface $X$.  The intersection
divisor $X\cap L$ has $\ell$ distinct points at which the intersection
multiplicities are prescribed by
$\lambda_1,\ldots,\lambda_\ell$, and the remaining $d-|\lambda|$
intersection points are simple.  In addition, $L$ is required to meet a
general projective subspace of codimension $i+1$.  The dimension condition
above makes the expected number of such lines finite.
For example, the classical formulas
\[
  \Pl_{(2,2);0}(d)
  =
  \frac12 d(d-2)(d-3)(d+3),
  \qquad
  \Pl_{(3);0}(d)
  =
  3d(d-2),
\]
give, respectively, the number of bitangent lines and flex lines to a
general degree $d$ plane curve.

In \cite{feher2023plucker} it is shown that the Pl\"ucker numbers are
polynomials in $d$, and several structural properties of these
polynomials are studied.  They can be computed from equivariant cohomology
classes of coincident root strata.  Let $\Pol^d(\mathbb C^2)$ denote the
space of binary forms of degree $d$.  For $d\ge |\lambda|$, define
%\[
%  Y_\lambda(d)
%  :=
%  \left\{
%  f\in \Pol^d(\mathbb C^2)
%  :
%  f=
%  \prod_{a=1}^{\ell} L_a^{\lambda_a}
%  \prod_{b=1}^{d-|\lambda|} M_b
%  \right\},
%\]
%where $L_a,M_b:\mathbb C^2\to\mathbb C$ are nonzero linear forms and all
%of them are pairwise non-proportional.  Thus
$Y_\lambda(d)$ to be the set of 
binary forms whose roots have exactly the prescribed multiplicities
$\lambda_1,\ldots,\lambda_\ell$, with all remaining roots simple.  Its
closure $\overline Y_\lambda(d)$ allows further collisions of roots. If $M=|\tilde\lambda|$, then the equivariant cohomology class of this
coincident root stratum has the Schur expansion
\begin{equation}\label{plucker}
    [\overline Y_\lambda(d)]
  =
  \sum_{|\tilde\lambda|/2 \le i \le |\tilde\lambda|}P_{\lambda;i}(d)\,s_{i,|\tilde\lambda|-i},
\end{equation}
Thus the Pl\"ucker polynomials appear as the Schur coefficients of
$\bigl[\overline Y_\lambda(d)\bigr]$.
We shift to the natural threshold $N=|\lambda|$, set
\[
  Q_{\lambda;i}(x):=P_{\lambda;i}(N+x),
\]
and write
\[
  Q_{\lambda;i}(x)=\sum_{r\ge0} C_{\lambda;i,r}\binom xr.
\]
Conjecture~\ref{conj:shifted-plucker-binomial-lc} predicts that the
coefficient sequence $(C_{\lambda;i,r})_r$ is nonnegative and
log-concave for every admissible pair $(\lambda,i)$.  By
Theorem~\ref{thm:binomial-lc-implies-lc}, this implies the weaker shifted
value log-concavity stated as
Conjecture~\ref{conj:shifted-plucker-value-lc}.  We give several small
examples and a proof strategy based on the recursive formula for
$[\overline Y_\lambda(d)]$; the proof is not complete, because the
missing step is a positive transition formula strong enough to preserve
log-concavity in the shifted binomial basis.

Subsection~\ref{subsec:csm-root-strata} further extends this viewpoint to equivariant
Chern--Schwartz--MacPherson classes of open coincident root strata, suggesting
that the shifted-binomial positivity and log-concavity phenomena may belong to
a broader theory of characteristic classes of root-collision strata.

\subsection{Methodology: orchestrating AI tools and human insight} The paper is a case study in orchestrating AI tools with human mathematical insight.  By \emph{orchestration} we mean that we, the human mathematicians, did not treat one system as a monolithic black box; rather, different systems were assigned different mathematical roles, and their outputs were repeatedly translated into conjectures, bases, identities, and proof obligations. The AI tools we used are the following. 
\begin{enumerate}
    \item AlphaEvolve \cite{AlphaEvolve} by Google DeepMind, which is an evaluator-guided evolutionary coding agent designed to improve candidate programs against exact tests; in this project it served first as a formula-finding engine and later as a detective for structural recurrences.  
    \item  ChatGPT 5.5 Pro was used as an interactive mathematical assistant for symbolic derivations and conjecture refinement. It produced closed formulas for the Chern classes $c_k(n,d)$, leading to a proof of a strong form of the polynomiality conjecture. It also supplied the proof of the $j=0$ case of the rank-two binomial log-concavity conjecture.  
    \item Google DeepMind's Co-Mathematician agent \cite{co-math} was used in the final stage of the project; it successfully provided proofs of the rank-two positivity for $j=2$ and log-concavity theorem for $j=1$ which were not available for GPT 5.5 Pro.
\end{enumerate}

The chronology of the project was as follows.  AlphaEvolve first suggested
the closed formula for $c_2(n,d)$ in Proposition~\ref{prop:c2}, together
with structural patterns in the cases $k=3,4$.  These patterns were then
used as input for ChatGPT~5.5 Pro, which provided the factorial-moment
identity of Lemma~\ref{lem:factorial-moments} and the proof of the strong
binomiality Theorem \ref{thm:strong-binomiality}. This settled the Feh\'er--Juh\'asz binomial-form
conjecture for the coefficients of $c_k(n,d)$, in a stronger form than
originally predicted.  ChatGPT~5.5 Pro also proved the parallel
$K$-theoretic analogue Theorem \ref{thm:NN-binomiality}, whose diagonal coefficients recover the
same Chern classes.

The next step was human: in rank two we changed basis from the Schur
coefficients $A_{k,j}(d)$ to their binomial-basis coefficients
$B_{k,j,r}$, and formulated the binomial positivity and binomial
log-concavity conjectures, Conjectures~\ref{conj:A} and~\ref{conj:B},
based on exact computations.  ChatGPT~5.5 Pro proved
Conjecture~\ref{conj:A} for $j=0,1$, and
Conjecture~\ref{conj:B} for $j=0$.  It did not prove the first remaining
cases, namely Conjecture~\ref{conj:A} for $j=2$ and
Conjecture~\ref{conj:B} for $j=1$.  In the latter case, the stronger
real-rootedness statement is false, so a different proof strategy was
needed.

A series of AlphaEvolve experiments, described in
Section~\ref{sec:detective}, then narrowed the search space for
Conjecture~\ref{conj:A}.  These experiments progressively isolated the
correct local recurrence, the appropriate residual terms, and finally the
lower-mode operator basis.

Co-Mathematician then produced proof drafts for the two remaining first
cases.  After verification and rewriting, these became the proofs of
Conjecture~\ref{conj:A} for $j=2$ and Conjecture~\ref{conj:B} for
$j=1$ explained in Section~\ref{sec:detective-proofs}.

We also investigated the analogous shifted-binomial log-concavity
conjecture for Pl\"ucker coefficients.  Those attempts did not lead to a
complete proof.  The ideas produced by Co-Mathematician are collected in
Section~\ref{sec:plucker-shifted-binomial-lc} as a proof strategy and as a
record of the present obstruction.

Thus the project benefited from a division of work among systems with
different strengths: conjecture formation, recurrence discovery, symbolic
derivation, proof drafting, and finite verification.  The methodological
lesson is that AI tools can be used as a research ensemble, while the human
mathematician supplies the problem selection, the changes of perspective,
and the standards of proof.

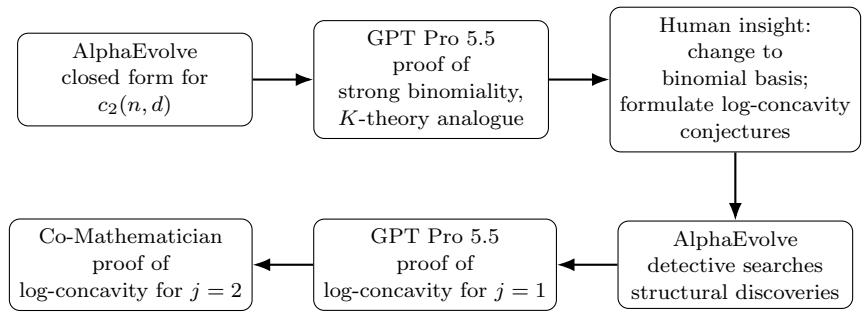
\begin{figure}[ht]
\centering
\begin{tikzpicture}[
  node distance=9mm and 8mm,
  box/.style={draw, rounded corners, align=center, minimum width=31mm, minimum height=11mm, font=\scriptsize},
  arrow/.style={-Latex, thick}
]
\node[box] (ae1) {AlphaEvolve\\ closed form for\\ $c_2(n,d)$};
\node[box, right=of ae1] (gpt) {GPT Pro  5.5\\ proof of \\ strong binomiality,\\ $K$-theory analogue};
\node[box, right=of gpt] (human) {Human insight:\\ change to\\ binomial basis;\\ formulate log-concavity \\ conjectures};
\node[box, below=of human] (ae2) {AlphaEvolve\\ detective searches\\ structural discoveries};
\node[box, left=of ae2] (comath) {GPT Pro  5.5\\ proof of \\ log-concavity for $j=1$};
\node[box, left=of comath] (out) {Co-Mathematician\\ proof of \\ log-concavity for $j=2$ };

\draw[arrow] (ae1) -- (gpt);
\draw[arrow] (gpt) -- (human);
\draw[arrow] (human) -- (ae2);
\draw[arrow] (ae2) -- (comath);
\draw[arrow] (comath) -- (out);
\end{tikzpicture}
\caption{The proof pipeline recorded in this paper.  The mathematician acts as a conductor: each AI tool contributes a distinct mathematical instrument.}
\label{fig:pipeline}
\end{figure}

Several mathematical questions remain. Higher-$j$ versions of the rank-two log-concavity theorem remain unresolved.  And although Theorem~\ref{thm:A1-solved} gives a native binomial-basis proof of positivity, a direct positive combinatorial interpretation of the coefficients $B_{k,2,r}$ is still open.

\subsection*{Acknowledgements}

G.B. thanks the Co-Mathematician and AlphaEvolve teams for providing early
access to their models, and for their valuable advice, support, and feedback
throughout this project. G.B. was supported by DFF 40296 grant of the Danish Independent Research
Fund.

\section{A binomiality theorem for symmetric powers}
This section describes the first mathematical phase of the workflow.  The starting point was the closed formula for $c_2(n,d)$ in Proposition~\ref{prop:c2}, first suggested experimentally by AlphaEvolve.  That formula exhibited exactly the binomial structure predicted by Feh\'er--Juh\'asz and led to the search for a general pattern for $c_3(n,d)$ and $c_4(n,d)$.  These patterns were then used as input for ChatGPT~5.5 Pro, which provided the factorial-moment
identity of Lemma~\ref{lem:factorial-moments} and the proof of the strong
binomiality Theorem \ref{thm:strong-binomiality}. The proof converts sums over weak compositions into binomial coefficients, and hence turns the Chern-root power sums into polynomials in the quantities $T_r(n,d)$. We then study the K-theory analogue where GPT Pro~5.5 successfully generalised Theorem \ref{thm:strong-binomiality} to the closed formula of Theorem \ref{thm:NN-binomiality}, whose diagonal coefficients recover the original Chern classes.

Let
\[
  A_{n,d}:=\{\alpha\in\Z_{\ge0}^n:|\alpha|=d\},\qquad
  y_\alpha:=\alpha_1x_1+\cdots+\alpha_nx_n.
\]
Then recall from the Introduction the Chern classes
\[
  \prod_{\alpha\in A_{n,d}}(1+y_\alpha)=\sum_{k\ge0}c_k(n,d).
\]
A two-parameter expression need not be polynomial in both parameters (also called separately polynomial function), even if it has a simple closed form.  For example, $d^n$ is a polynomial in $d$ for every fixed $n$, but after setting $d=2$ one obtains $2^n$, which is not a polynomial in $n$.  Thus polynomiality in the rank $n$ and the degree $d$ is already a genuine structural constraint.

\begin{conjecture}[Separate polynomiality, \cite{feher2025polynomiality}]\label{conj:polynomiality}
For each fixed $k$ and each partition $\lambda\vdash k$, the coefficient $f_\lambda(n,d)$ in
\[
  c_k(n,d)=\sum_{\lambda\vdash k}f_\lambda(n,d)e_\lambda
\]
is polynomial in $d$ for fixed $n$, and polynomial in $n$ for fixed $d$.
\end{conjecture}

The following conjecture is a stronger form of this polynomiality statement: it predicts not only polynomial dependence on the parameters separately, but polynomial dependence through a distinguished family of binomial coefficients.

\begin{conjecture}[Strong binomial form, \cite{feher2025polynomiality}]\label{conj:binomial-form}
For each fixed $k$ and each partition $\lambda\vdash k$, the coefficient $f_\lambda(n,d)$ in
\[
  c_k(n,d)=\sum_{\lambda\vdash k}f_\lambda(n,d)e_\lambda
\]
is a polynomial in finitely many binomial coefficients of the form $\binom{d+n+a}{n+b}$.
\end{conjecture}

The first nontrivial low-degree case already exhibits the special binomial structure.

\begin{proposition}[Closed formula for $c_2(n,d)$]\label{prop:c2}
One has
\[
  c_2(n,d)=\binom{d+n}{n+1}e_2+
  \left(
    \frac12\binom{d+n-1}{n}^2
    -\frac12\binom{d+n-1}{n}
    -\binom{d+n-1}{n+1}
  \right)e_1^2.
\]
This formula was  found experimentally by AlphaEvolve, recovering the result of \cite{ciliberto_zaidenberg2020onfanoschemes}.
\end{proposition}

\begin{proof}
Let
\[
  u=\binom{d+n-1}{n},\qquad v=\binom{d+n}{n+1}.
\]
For $m\ge1$ put
\[
  P_m(n,d):=\sum_{\alpha\in A_{n,d}}y_\alpha^m.
\]

A direct generating-function calculation gives
\[
  P_1(n,d)=\sum_{\alpha\in A_{n,d}}y_\alpha=u\,e_1
\]
and
\[
  P_2(n,d)=\sum_{\alpha\in A_{n,d}}y_\alpha^2=(2v-u)e_1^2-2v\,e_2.
\]
Newton's identity $2c_2=P_1^2-P_2$ now yields the formula.
\end{proof}

It turns out that the much stronger version of Conjecture \ref{conj:binomial-form} holds: the coefficients of the Chern classes are polynomials of an explicit short list of binomial coefficients $T_r(n,d)$ below. 

\begin{theorem}[Strong binomiality]\label{thm:strong-binomiality}
For every $k\ge0$, there is a universal polynomial
\[
  Q_k\in\Q[u_1,\ldots,u_k,z_1,\ldots,z_k]
\]
such that, for all $n,d$,
\[
  c_k(n,d)=Q_k\bigl(T_1(n,d),\ldots,T_k(n,d);e_1,\ldots,e_k\bigr),
\]
where
\[
  T_r(n,d):=\binom{d+n-1}{n+r-1},\qquad 1\le r\le k.
\]
Equivalently, $f_\lambda(n,d)\in\Q[T_1(n,d),\ldots,T_k(n,d)]$ for every $\lambda\vdash k$.  In particular, Conjecture \ref{conj:binomial-form} holds.
\end{theorem}

\begin{lemma}[Factorial moments of weak compositions]\label{lem:factorial-moments}
Let $i_1,\ldots,i_s$ be distinct indices and let $m_1,\ldots,m_s$ be positive integers.  Put $M=m_1+\cdots+m_s$ and write $a^{\underline m}=a(a-1)\cdots(a-m+1)$.  Then
\begin{equation}\label{eq:factorial-moments}
  \sum_{\alpha\in A_{n,d}}\prod_{j=1}^s \alpha_{i_j}^{\underline {m_j}}
  =
  \left(\prod_{j=1}^s m_j!\right)\binom{d+n-1}{n+M-1}.
\end{equation}
\end{lemma}

\begin{proof}
Taking generating functions in $d$ gives
\[
  \sum_{d\ge0}\left(\sum_{\alpha\in A_{n,d}}\prod_{j=1}^s \alpha_{i_j}^{\underline {m_j}}\right)t^d
  =
  \left(\prod_{j=1}^s\frac{m_j!t^{m_j}}{(1-t)^{m_j+1}}\right)(1-t)^{-(n-s)}.
\]
This is
\[
  \left(\prod_{j=1}^s m_j!\right)t^M(1-t)^{-(n+M)}.
\]
Taking the coefficient of $t^d$ proves the lemma.
\end{proof}

\begin{proof}[Proof of Theorem \ref{thm:strong-binomiality}]
Fix $m\ge1$.  We first show that
\begin{equation}\label{eq:P-m-in-ring}
  P_m(n,d)\in \Z[T_1(n,d),\ldots,T_m(n,d)]\otimes_{\Z}\Lambda_m,
\end{equation}
where $\Lambda_m$ denotes homogeneous symmetric functions of degree $m$.

Let $\mu=(\mu_1,\ldots,\mu_s)\vdash m$.  The coefficient of a monomial of type $\mu$ in $P_m$ is a multinomial factor times
\[
  M_\mu(n,d):=\sum_{\alpha\in A_{n,d}}\alpha_1^{\mu_1}\cdots\alpha_s^{\mu_s}.
\]
Using Stirling numbers of the second kind,
\[
  a^q=\sum_{r=1}^{q}\StirTwo{q}{r}a^{\underline r},
\]
and applying Lemma \ref{lem:factorial-moments}, we obtain
\begin{equation}\label{eq:mixed-moment}
  M_\mu(n,d)=
  \sum_{1\le r_j\le \mu_j}
  \left(\prod_{j=1}^s \StirTwo{\mu_j}{r_j}r_j!\right)
  T_{r_1+\cdots+r_s}(n,d).
\end{equation}
Thus each monomial-symmetric coefficient of $P_m$ lies in $\Z[T_1,\ldots,T_m]$.  Since the elementary products form a basis in each degree over $\Q$, $P_m$ can be written universally as a polynomial in $T_1,\ldots,T_m$ and $e_1,\ldots,e_m$.

Newton identities for the roots $y_\alpha$ give
\[
  k c_k=\sum_{j=1}^{k}(-1)^{j-1}c_{k-j}P_j,\qquad c_0=1.
\]
By induction on $k$, the right hand side lies in $\Q[T_1,\ldots,T_k][e_1,\ldots,e_k]$.  Dividing by $k$ proves the theorem.
\end{proof}
\begin{example} \label{ex:small-k-cknd} For small $k$ we have
\[\begin{aligned}
c_1(n,d)&=T_1e_1,\\[3pt]
c_2(n,d)&=
\frac{T_1^2-T_1-2T_2}{2}e_{1}^2
+(T_1+T_2)e_2,\\[3pt]
c_3(n,d)&=
\frac{T_1^3-3T_1^2-6T_1T_2+2T_1+12T_2+12T_3}{6}e_{1}^3\\
&\quad+
\left(T_1^2+T_1T_2-T_1-5T_2-4T_3\right)e_1e_2\\
&\quad+
\left(T_1+3T_2+2T_3\right)e_3.
\end{aligned}\]
\end{example}
\medskip
\begin{remark} \label{strong-binom4schur}
Theorem\ref{thm:strong-binomiality} immediately implies that the Schur coefficents of $c_k(n,d)$ are also polynomials in the binomial expressions $T_r$.
\end{remark}

\subsection{The case of $d=2$: the Laksov--Lascoux--Thorup formula}
\smallskip

For $d=2$,
\[
  C_n(x):=c(\Sym^2\C^n)=\prod_{i=1}^n(1+2x_i)\prod_{1\le i<j\le n}(1+x_i+x_j).
\]

The case of $d=2$ can be calculated from \eqref{eq:mixed-moment} but ChatGPT~5.5 Pro gives the following shorter proof:

\begin{theorem}[ \cite{LLT} The $d=2$ determinant]\label{thm:d2det}
Let $\lambda\vdash k$ with $\ell(\lambda)\le n$.  Then
\[
  c_k(n,2)=\sum_{\lambda\vdash k,\ \ell(\lambda)\le n}A_\lambda(n)s_\lambda(x),
\]
where
\begin{equation}\label{eq:d2-coeff}
  A_\lambda(n)=2^{k-\binom n2}
  \det\left[\binom{2n-2j+1}{\lambda_i+n-i}\right]_{1\le i,j\le n}.
\end{equation}
\end{theorem}

\begin{proof}
Let $h_r$ be the complete symmetric functions, with $h_r=0$ for $r<0$.  For a strictly increasing sequence $I=(i_1<\cdots<i_n)$ put
\[
  s_I=\det(h_{i_q-p+1})_{1\le p,q\le n}.
\]
Set $A_{r,q}=2^r\binom{2n-2q+1}{r}$ and $B_{p,r}=h_{r-p+1}$.  By Cauchy--Binet,
\[
  \det(BA)=\sum_I s_I\det(A_I).
\]
Using
\[
  h_m=\sum_{a=1}^n \frac{x_a^{m+n-1}}{\prod_{b\ne a}(x_a-x_b)},
\]
one obtains
\[
  (BA)_{p,q}=\sum_{a=1}^n
  \frac{x_a^{n-p}(1+2x_a)^{2n-2q+1}}{\prod_{b\ne a}(x_a-x_b)}.
\]
Taking determinants,
\[
  \det(BA)=(-1)^{\binom n2}
  \frac{\det((1+2x_a)^{2n-2q+1})_{a,q}}{\prod_{a<b}(x_a-x_b)}.
\]
With $u_a=1+2x_a$,
\[
  \det(u_a^{2n-2q+1})_{a,q}
  =\left(\prod_a u_a\right)\prod_{a<b}(u_a^2-u_b^2).
\]
Since $u_a^2-u_b^2=4(x_a-x_b)(1+x_a+x_b)$, we get
\[
  \det(BA)=(-1)^{\binom n2}2^{n(n-1)}C_n(x).
\]
Reindexing $I$ by partitions gives \eqref{eq:d2-coeff}.
\end{proof}

\begin{example}
For small $k$,
\[
  c_1(n,2)=(n+1)s_1,
\]
\[
  c_2(n,2)=\frac{(n-1)(n+2)}2s_2+\frac{(n+1)(n+2)}2s_{11},
\]
\[
  c_3(n,2)=\frac{(n-2)(n-1)(n+3)}6s_3+
  \frac{(n+2)(n^2+n-3)}3s_{21}
  +\frac{(n+1)(n+2)(n+3)}6s_{111}.
\]
\end{example}

\medskip
\subsection{A $K$-theoretic generalization.}
The following normalized multiplicative generalization
was proved in this project by GPT Pro~5.5.  It is useful to regard it as a $K$-theoretic or finite-difference refinement of Theorem~\ref{thm:strong-binomiality}: the variables $X_i$ are multiplicative Chern-root variables, while the substitution $X_i=1-\zeta u_i$ extracts the additive Chern-root information order by order in $\zeta$.

Let
\[
  M_{n,d}(y,X)=\prod_{\alpha\in A_{n,d}}(1+yX^\alpha),
  \qquad X^\alpha=X_1^{\alpha_1}\cdots X_n^{\alpha_n}.
\]

This is the $K$-theoretic total Chern class, or equivariant motivic Chern class of the representation $\Sym^d(\mathbb C^n)$. Using a substitution we can compare it with the cohomological Chern class:

Define
\[
  \NN_{n,d}(z,\zeta,u):=
  (1+z)^{|A_{n,d}|}
  M_{n,d}\left(-\frac{z}{1+z},\,X_j\mapsto 1-\zeta u_j\right).
\]
Equivalently,
\begin{equation}\label{eq:NN-def}
  \NN_{n,d}(z,\zeta,u)=
  \prod_{\alpha\in A_{n,d}}
  \left(1+z\left(1-\prod_{j=1}^n(1-\zeta u_j)^{\alpha_j}\right)\right).
\end{equation}
Then the cohomological Chern class is the diagonal:
\begin{proposition}[Diagonal recovery of the Chern classes]\label{prop:NN-diagonal}
For every $k\ge0$,
\[
  [z^k\zeta^k]\NN_{n,d}(z,\zeta,u)=c_k(n,d).
\]
\end{proposition}

\begin{proof}
For each $\alpha\in A_{n,d}$,
\[
  1-\prod_{j=1}^n(1-\zeta u_j)^{\alpha_j}
  =\zeta\sum_{j=1}^n\alpha_j u_j+O(\zeta^2).
\]
To obtain $z^k\zeta^k$, one must choose the $z$-term from exactly $k$ factors in \eqref{eq:NN-def}.  Since each chosen factor contributes at least one power of $\zeta$, the total $\zeta$-degree $k$ forces us to take only the linear part displayed above.  Therefore
\[
  [z^k\zeta^k]\NN_{n,d}
  =\sum_{\substack{S\subseteq A_{n,d}\\ |S|=k}}
  \prod_{\alpha\in S}\left(\sum_j\alpha_j u_j\right),
\]
which is precisely the degree-$k$ homogeneous component of
\[
  \prod_{\alpha\in A_{n,d}}
  \left(1+\sum_j\alpha_j u_j\right).
\]
\end{proof}
 Using this we can see that the following theorem generalizes Theorem \ref{thm:strong-binomiality}:
\begin{theorem}[Binomiality for the $\NN$-analogue]\label{thm:NN-binomiality}
For every fixed pair $q,m\ge0$, the coefficient
\[
  [z^q\zeta^m]\NN_{n,d}(z,\zeta,u)
\]
is a symmetric homogeneous polynomial of degree $m$ in $u_1,\ldots,u_n$.  If
\[
  [z^q\zeta^m]\NN_{n,d}
  =\sum_{\lambda\vdash m}F_{q,m,\lambda}(n,d)e_\lambda(u),
\]
then
\[
  F_{q,m,\lambda}(n,d)\in\Q[T_1(n,d),\ldots,T_m(n,d)].
\]
Moreover,
\[
  [z^q\zeta^m]\NN_{n,d}=0\qquad\text{if }q>m.
\]
\end{theorem}

\begin{proof}
For $\alpha\in A_{n,d}$ put
\[
  A_\alpha(\zeta,u):=1-\prod_{j=1}^n(1-\zeta u_j)^{\alpha_j}.
\]
Then \eqref{eq:NN-def} says
\[
  \NN_{n,d}=\prod_{\alpha\in A_{n,d}}(1+zA_\alpha),
\]
and hence
\[
  [z^q]\NN_{n,d}=e_q(A_\alpha:\alpha\in A_{n,d}).
\]
Newton identities express this elementary symmetric function universally in the power sums
\[
  P_s^{\NN}:=\sum_{\alpha\in A_{n,d}}A_\alpha^s,
  \qquad 1\le s\le q.
\]
Now
\[
  A_\alpha(\zeta,u)=
  \sum_{\beta\ne0}(-1)^{|\beta|+1}
  \binom{\alpha}{\beta}\zeta^{|\beta|}u^\beta,
  \qquad
  \binom{\alpha}{\beta}:=\prod_{j=1}^n\binom{\alpha_j}{\beta_j}.
\]
Thus the coefficient of $\zeta^m u^\mu$ in $P_s^{\NN}$ is a finite linear combination of sums of the form
\[
  \sum_{\alpha\in A_{n,d}}
  \prod_j
  \binom{\alpha_j}{\beta_j^{(1)}}
  \binom{\alpha_j}{\beta_j^{(2)}}\cdots
  \binom{\alpha_j}{\beta_j^{(s)}}.
\]
For fixed $j$, the product of binomial coefficients is a polynomial in $\alpha_j$ and can be expanded in the falling-factorial basis.  The total falling-factorial degree is at most $m$.  Applying Lemma~\ref{lem:factorial-moments} shows that every such sum lies in
\[
  \Q[T_1(n,d),\ldots,T_m(n,d)].
\]
Therefore the same is true for every coefficient of each $P_s^{\NN}$, and Newton identities imply the assertion for $[z^q]\NN_{n,d}$.  Extracting $\zeta^m$ proves the coefficient statement.  Finally, each $A_\alpha$ has $\zeta$-order at least one, so a term of $z$-degree $q$ has $\zeta$-order at least $q$.  Hence $[z^q\zeta^m]\NN_{n,d}=0$ for $q>m$.
\end{proof}

\begin{example}[Small $\NN$-coefficients]\label{ex:NN-small}
Write $T_r=T_r(n,d)$.  The first coefficients are
\[
  [z\zeta]\NN_{n,d}=T_1e_1,
\]
\[
  [z\zeta^2]\NN_{n,d}=-T_2e_{1,1}+T_2e_2,
\]
\[
  [z^2\zeta^2]\NN_{n,d}=
  \frac{T_1^2-T_1-2T_2}{2}e_{1,1}+(T_1+T_2)e_2,
\]
\[
  [z\zeta^3]\NN_{n,d}=T_3e_{1,1,1}-2T_3e_{2,1}+T_3e_3,
\]
\begin{align*}
  [z^2\zeta^3]\NN_{n,d}={}&
  (3T_3+2T_2-T_1T_2)e_{1,1,1}\\
  &+(T_1T_2-6T_3-5T_2)e_{2,1}
  +(3T_3+3T_2)e_3,
\end{align*}
\begin{align*}
  [z^3\zeta^3]\NN_{n,d}={}&
  \frac{T_1^3-3T_1^2-6T_1T_2+12T_3+12T_2+2T_1}{6}e_{1,1,1}\\
  &+(T_1^2+T_1T_2-4T_3-5T_2-T_1)e_{2,1}
  +(2T_3+3T_2+T_1)e_3.
\end{align*}

Notice that the coefficient of $[z^k\zeta^k]$ agrees with $c_k(n,d)$ in \ref{ex:small-k-cknd}, as the diagonal recovery Proposition \ref{prop:NN-diagonal} claims.
\end{example}

\subsection{Leading terms and applications}
Formula \eqref{eq:mixed-moment} immediately implies that the $d$ leading coefficients of $c_k(n,d)$ in the elementary symmetric basis are
\[ [e_\lambda]{c_k(n,d)}
  =
  \frac{1}{\prod_{i\ge1}m_i!\,((d-1)!)^{\ell(\lambda)}}
  n^{(d-1)\ell(\lambda)}
  +O\left(n^{(d-1)\ell(\lambda)-1}\right).
\]

This was conjectured in \cite[Conj 3.10]{feher2025polynomiality}. 

The leading terms for the Schur basis were obtained \cite{feher2025polynomiality}. That result led to the generalization of a theorem of Manivel of the degree of hypersurfaces containing a projective  $r$-space. Interestingly, the leading terms for the Schur basis are easier to calculate and their degree is uniform $n|\lambda|$. For the elementary symmetric basis the degrees can be much smaller due to a surprising cancellation.

The result on the leading coefficients of $c_k(n,d)$ in the elementary symmetric basis is more subtle, but we are not aware of any enumerative applications.

\subsection{Euler classes}
The Schur coefficients of the top Chern class or Euler class 

\begin{equation}\label{eq:etop}
  E_{n,d}(x):=c_{\rm top}(n,d)=
  \prod_{\substack{\alpha\in\Z_{\ge0}^n\\|\alpha|=d}}
  (\alpha_1x_1+\cdots+\alpha_nx_n).
\end{equation}
also has enumerative applications as it is explained e.g. in \cite{feher2025polynomiality}.  For $n=2$ let 
\[
c_{\mathrm{top}}(2,d)
=
\sum_{b=0}^{\left\lfloor (d+1)/2\right\rfloor}
\gamma_{d,b}\,
s_{(d+1-b,b)}(x_1,x_2).
\]
be the Schur expansion of the Euler class. Then ChatGPT 5.5 Pro gives

\[
\gamma_{d,b}
=
\sum_{r=b-1}^{d}
(-1)^{r-b}
\binom{r+1}{b}
d^{\,d+1-r}
\left[{d+1\atop d+1-r}\right],
\]

which is somewhat simpler than the formula of \cite{feher2025polynomiality}. This simplifies the formula of \cite{feher2025polynomiality} for the Euler characteristics of Fano varieties of lines of hypersurfaces. Here square brackets denote the unsigned Stirling numbers of the first kind.
For higher $n$ ChatGPT 5.5 Pro did not provide closed formulas.
\bigskip
\begin{remark}
Similar enumerative applications can be given for the $K$-theory Chern class, leading to calculation of Hilbert polynomials and Todd classes of Fano varieties. These applications will be published elsewhere.
\end{remark}

\section{A rank-two binomial refinement}

This section introduces the refined rank-two phenomenon that became the central target of the later AI-assisted proof stage.
We study the $n=2$, rank two case and write $c_k(2,d)$ in the Schur basis:
\[
  c_k(2,d)=\sum_{0\le j\le\lfloor k/2\rfloor}A_{k,j}(d)s_{(k-j,j)}(x_1,x_2).
\]
We then change coordinates once more by expanding each Schur coefficient $A_{k,j}(d)$ in the binomial basis of the degree parameter $d$:
\[
  A_{k,j}(d)=\sum_{r\ge0} B_{k,j,r}\binom{d}{r}.
\]
This change of basis, which was previously unknown in the literature and was motivated by Theorem \ref{thm:strong-binomiality}, revealed positivity and log-concavity patterns that were invisible in the original $d$-power basis.  

\begin{definition}\label{def:binom} We will call a polynomial
\[
  f(d)=\sum_{r\ge0} b_r\binom{d}{r}
\]
\emph{binomially positive} if $b_r\ge0$ for all $r$.  We will call it \emph{binomially log-concave} if the nonzero sequence $(b_r)_r$ has no internal zeroes and is log-concave, that is, $b_r^2\ge b_{r-1}b_{r+1}$.
\end{definition}
The following rank-two conjectures predict that, for each fixed $k$ and $j$, the Schur coefficient $A_{k,j}(d)$ is binomially positive and binomially log-concave as a polynomial in $d$.

\begin{conjecture}[Rank-two binomial positivity]\label{conj:A}
For all $k,j,r$ with $0\le j\le\lfloor k/2\rfloor$, one has
\[
  B_{k,j,r}\ge0.
\]
Equivalently, for each fixed $k$ and $j$, the polynomial $A_{k,j}(d)$ is binomially positive.
\end{conjecture}

\begin{conjecture}[Rank-two binomial log-concavity]\label{conj:B}
For fixed $k$ and $j$, the nonzero sequence $(B_{k,j,r})_r$ is log-concave:
\[
  B_{k,j,r}^2\ge B_{k,j,r-1}B_{k,j,r+1}.
\]
Equivalently, $A_{k,j}(d)$ is binomially log-concave as a polynomial in $d$.
\end{conjecture}

ChatGPT~5.5 Pro proved the first accessible cases, namely binomial positivity for $j=0,1$ and binomial log-concavity for $j=0$, using fixed-point-free permutations with prescribed cycle defect.  The remaining cases were then isolated as the precise problems for the later AlphaEvolve and Co-Mathematician stages.

\subsection{Binomial log-concavity implies ordinary log-concavity}

We first prove that binomial log-concavity implies ordinary log-concavity of the values $A_{k,j}(d)$ in the degree parameter. This is a well-known result, but here we give a proof to keep our study self-contained.

\begin{theorem}[Binomial log-concavity implies ordinary log-concavity]\label{thm:binomial-lc-implies-lc}
Let $f(d)$ be a binomially positive and binomially log-concave function, as in Definition \ref{def:binom}.
Then the nonzero sequence $(f(d))_{d\ge0}$ is log-concave:
\[
  f(d)^2\ge f(d-1)f(d+1)
\]
for all $d$ in its nonzero range.
\end{theorem}

\begin{proof}
Let the support of $(b_r)_r$ be the interval $[m,M]$.  We first prove the result in the case $m=0$ and $b_r>0$ for $0\le r\le M$.  Put $b_{M+1}=0$ and
\[
  q_r:=\frac{b_{r+1}}{b_r}\qquad(0\le r\le M).
\]
The log-concavity of $(b_r)_r$ is equivalent to
\[
  q_0\ge q_1\ge \cdots \ge q_M.
\]

Write
\[
  f_d:=f(d)=\sum_{r=0}^{M}b_r\binom dr.
\]
By Pascal's identity,
\[
  f_{d+1}=f_d+g_d,
  \qquad
  g_d:=\sum_{r=0}^{M}b_{r+1}\binom dr.
\]
Define a probability distribution on $\{0,\ldots,M\}$ by
\[
  p_d(r):=\frac{b_r\binom dr}{f_d}.
\]
Then
\[
  \frac{g_d}{f_d}
  =
  \sum_{r=0}^{M}q_r p_d(r).
\]
We claim that the right-hand side is nonincreasing in $d$.  Indeed, for $0\le r\le d-1$,
\[
  \frac{p_d(r)}{p_{d-1}(r)}
  =
  \frac{f_{d-1}}{f_d}\frac{\binom dr}{\binom{d-1}{r}}
  =
  \frac{f_{d-1}}{f_d}\frac{d}{d-r},
\]
which is an increasing function of $r$.  Thus $p_d$ is larger than $p_{d-1}$ in the monotone likelihood-ratio order.  Since $q_r$ is nonincreasing in $r$, we obtain
\[
  \sum_{r=0}^{M}q_r p_d(r)
  \le
  \sum_{r=0}^{M}q_r p_{d-1}(r).
\]
Equivalently,
\[
  \frac{g_d}{f_d}\le \frac{g_{d-1}}{f_{d-1}}.
\]
Hence
\[
  \frac{f_{d+1}}{f_d}
  =
  1+\frac{g_d}{f_d}
  \le
  1+\frac{g_{d-1}}{f_{d-1}}
  =
  \frac{f_d}{f_{d-1}},
\]
which is precisely
\[
  f_d^2\ge f_{d-1}f_{d+1}.
\]

It remains to remove the assumption $m=0$.  If the support starts at $m>0$, replace the zero prefix by
\[
  b_r^{(\varepsilon)}:=b_m\varepsilon^{m-r}
  \qquad(0\le r<m),
\]
and set $b_r^{(\varepsilon)}:=b_r$ for $m\le r\le M$.  For sufficiently small $\varepsilon>0$, the ratio sequence
\[
  \frac{b_{r+1}^{(\varepsilon)}}{b_r^{(\varepsilon)}}
\]
is still nonincreasing, so the first part of the proof applies to
\[
  f^{(\varepsilon)}(d)
  :=
  \sum_{r=0}^{M}b_r^{(\varepsilon)}\binom dr.
\]
Thus $(f^{(\varepsilon)}(d))_d$ is log-concave.  Letting $\varepsilon\to0$ gives the desired log-concavity inequalities for $f(d)$, since these inequalities are closed under limits.
\end{proof}

Consequently, Conjectures~\ref{conj:A} and~\ref{conj:B} imply that, for every fixed $k$ and $j$, the sequence
\[
  A_{k,j}(0),A_{k,j}(1),A_{k,j}(2),\ldots
\]
is log-concave after deleting its initial zeroes.

\subsection{Proof of binomial log-concavity for $j=0$}
We start with a closed formula using second order Stirling coefficients. Specialize to $n=2$.  Put $x_1=u$, $x_2=1$, and
\[
  C_{k,d}(u):=c_k(2,d)(u,1)
  =e_k(i+(d-i)u:0\le i\le d).
\]
Since $s_{(k-j,j)}(u,1)=u^j+\cdots+u^{k-j}$,
\begin{equation}\label{eq:A-extraction}
  A_{k,j}(d)=[u^j](1-u)C_{k,d}(u).
\end{equation}
From
\[
  \prod_{i=0}^{d}(X+i+(d-i)u)=
  (1-u)^{d+1}\left(\frac{X+du}{1-u}\right)^{\overline{d+1}},
\]
we get
\begin{equation}\label{eq:Ckd-stirling}
  C_{k,d}(u)=
  \sum_{s=0}^{k}
  \stir{d+1}{d+1-k+s}
  \binom{d+1-k+s}{s}d^s u^s(1-u)^{k-s}.
\end{equation}
Indeed, here $y^{\overline{m}}$ denotes the rising factorial
\[
  y^{\overline{m}}=y(y+1)\cdots(y+m-1),
\]
and
\[
  X+i+(d-i)u
  =
  (1-u)\left(\frac{X+du}{1-u}+i\right).
\]
Thus
\begin{equation}\label{eq:Akj-stirling}
  A_{k,j}(d)=
  \sum_{s=0}^{j}(-1)^{j-s}\binom{k-s+1}{j-s}d^s
  \binom{d+1-k+s}{s}
  \stir{d+1}{d+1-k+s}.
\end{equation}

Let $\Dder(N,K)$ be the number of fixed-point-free permutations of an $N$-element set with defect $K$, i.e. with $N-K$ cycles.  We put $\Dder(N,K)=0$ outside the natural range and $\Dder(0,0)=1$.

\begin{lemma}\label{lem:D-recurrence}
The numbers $\Dder(N,K)$ satisfy
\[
  \Dder(N,K)=(N-1)(\Dder(N-1,K-1)+\Dder(N-2,K-1)).
\]
\end{lemma}

\begin{proof}
Look at the cycle containing a distinguished letter.  If it has length at least three, deleting the distinguished letter gives a fixed-point-free permutation on $N-1$ letters with defect $K-1$, and there are $N-1$ insertion positions.  If it lies in a transposition, deleting the transposition gives a fixed-point-free permutation on $N-2$ letters with defect $K-1$, and there are $N-1$ choices for the partner.
\end{proof}

\begin{theorem}[The cases $j=0,1$ of binomial positivity]\label{thm:rank-two-known-A}
For all $k,r$,
\[
  B_{k,0,r}\ge0,\qquad B_{k,1,r}\ge0.
\]
More precisely,
\[
  B_{k,0,r}=\Dder(r,k)+\Dder(r+1,k),
\]
and
\begin{align*}
B_{k,1,r}={}&r(r+2-k)\Dder(r+1,k-1)
+r(3r+2-3k)\Dder(r,k-1)\\
&+(3r^2-3kr-2r+k+1)\Dder(r-1,k-1)\\
&+(r-1)(r-k-1)\Dder(r-2,k-1).
\end{align*}
\end{theorem}

\begin{proof}
For $j=0$, \eqref{eq:Akj-stirling} gives $A_{k,0}(d)=\stir{d+1}{d+1-k}$.  A permutation of $d+1$ letters with $d+1-k$ cycles has a support of non-fixed letters.  If the support has size $N$, the induced fixed-point-free permutation has $N-k$ cycles and is counted by $\Dder(N,k)$.  Hence
\[
  \stir{d+1}{d+1-k}=\sum_N\Dder(N,k)\binom{d+1}{N}.
\]
Since $\binom{d+1}{N}=\binom dN+\binom d{N-1}$, this gives the displayed formula for $B_{k,0,r}$.

For $j=1$, \eqref{eq:Akj-stirling} gives
\[
  A_{k,1}(d)=d(d+2-k)\stir{d+1}{d+2-k}-(k+1)\stir{d+1}{d+1-k}.
\]
Expanding both Stirling numbers by supports of non-fixed letters and using Lemma~\ref{lem:D-recurrence}, one rewrites the coefficient of $\binom dr$ in the displayed formula as the stated expression.  The required binomial-basis rewritings are
\[
  d\binom ds=s\binom ds+(s+1)\binom d{s+1}
\]
and
\[
  d^2\binom ds=s^2\binom ds+(s+1)(2s+1)\binom d{s+1}+(s+1)(s+2)\binom d{s+2}.
\]
Each active coefficient in the displayed formula is nonnegative.  For example, if $\Dder(r,k-1)\ne0$, then $r\ge k$, so $r(3r+2-3k)\ge0$; the other terms are checked similarly.
\end{proof}

\begin{theorem}[The case $j=0$ of binomial log-concavity]\label{thm:rank-two-known-B}
For fixed $k$, the nonzero sequence $(B_{k,0,r})_r$ is log-concave.
\end{theorem}

\begin{proof}
Let $P_k(t)=\sum_N\Dder(N,k)t^N$.  Lemma~\ref{lem:D-recurrence} gives
\[
  P_k(t)=t^2\frac{d}{dt}\bigl((1+t)P_{k-1}(t)\bigr),\qquad P_0(t)=1.
\]
By induction and Rolle's theorem, all roots of $P_k(t)$ are real and nonpositive.  Since
\[
  \sum_rB_{k,0,r}t^r=\frac{1+t}{t}P_k(t),
\]
this generating polynomial also has only real nonpositive roots.  Newton's inequalities imply log-concavity of its nonzero coefficient sequence.
\end{proof}

\begin{remark}
The same real-rootedness mechanism does not prove Conjecture~\ref{conj:B} for $j>0$.  Already for $j=1$ we have
\[
  \sum_rB_{2,1,r}t^r=t+4t^2+6t^3+3t^4=t(1+4t+6t^2+3t^3),
\]
and the cubic factor has non-real roots.
\end{remark}

The first values for Conjecture~\ref{conj:A} are:
\begin{align*}
A_{4,2}(d)={}&27\binom d3+212\binom d4+660\binom d5+1000\binom d6+735\binom d7+210\binom d8,\\
A_{5,2}(d)={}&320\binom d4+3375\binom d5+14010\binom d6+29575\binom d7
+33740\binom d8+19845\binom d9+4725\binom d{10},\\
A_{6,2}(d)={}&3250\binom d5+44004\binom d6+233604\binom d7+646492\binom d8
+1022217\binom d9\\
&+932400\binom d{10}+457380\binom d{11}+93555\binom d{12}.
\end{align*}

\section{An AlphaEvolve detective story for Conjecture \ref{conj:A}}
\label{sec:detective}

This section reports a series of AlphaEvolve experiments in search for approximate identities and leading terms for the coefficients $B_{k,2,r}$. These led to the proof of
Conjecture~\ref{conj:A} for $j=2$ and Conjecture~\ref{conj:B} for $j=1$. The role of AlphaEvolve was not to prove a
statement, but to propose candidate recurrences for the exact integer table
of coefficients $B_{k,2,r}$.  The displayed approximate identities below
are therefore not used as formal arguments.  They are included because they
explain how the correct mathematical language was found: the final proof in
Section~\ref{sec:detective-proofs} is a binomial-basis induction involving
lower modes and multiplication by linear factors.

We use the reindexed array
\[
  M_{k,m}:=B_{k,2,2k-m},\qquad r=2k-m.
\]
Thus $m=0$ corresponds to the top edge $r=2k$, and increasing $m$
moves downward in the binomial-basis degree $r$.  All arrays are extended
by zero outside their natural supports.  We also keep the defect-derangement
notation $\Dder(N,K)$ introduced above.

In this section, the symbol $\approx$ has only an experimental meaning:
the proposed formula gave a good fit on finite rectangles of exactly
computed values, but it is not asserted to be an identity.

\subsection{A first local ansatz.}
The first useful pattern was that $M_{k,m}$ should be determined mainly
from its two nearest predecessors in the previous $k$-row:
\[
  M_{k-1,m-1},\qquad M_{k-1,m}.
\]
The simplest successful ansatz had the form
\[
  M_{k,m}
  \approx
  Q_1(k,m)M_{k-1,m-1}
  +
  Q_2(k,m)M_{k-1,m}
  +
  \text{a short }\Dder\text{-source}.
\]
One compact candidate was
\begin{equation}\label{eq:oracle-fed-merged}
\begin{aligned}
M_{k,m} \approx {}&
m\,M_{k-1,m-1}+(2k-m-1)M_{k-1,m}  \\
&+2\Dder(r,k)
+4k\,\Dder(r-1,k-1)
+3k(k-1)\Dder(r-2,k-2)  \\
&+\frac{4}{3}k(k-1)(k-2)\Dder(r-3,k-3),
\end{aligned}
\end{equation}
where $r=2k-m$.  The source term in \eqref{eq:oracle-fed-merged} is not
the right one, but the recurrence revealed an important feature: the
dependence on $k$ should be local in the $(k,m)$-lattice.

\subsection{A symmetric residual.}
The next useful observation came from subtracting a simple two-neighbour
operator.  Define the operator
\[
  W(F)(k,m):=(2k-m-1)\bigl(F(k-1,m)+F(k-1,m-1)\bigr),
\]
and set
\[
  E_{k,m}:=(I-W)M_{k,m}.
\]
The residual $E$ appeared to satisfy a simpler recurrence:
\begin{equation}\label{eq:ward-defect-merged}
  E_{k,m}
  \approx
  (2k-m-1)E_{k-1,m}
  +(2k-m-1)E_{k-1,m-1}
  +m\,\Dder(r+1,k).
\end{equation}
Equivalently, this suggests the schematic second-order relation
\[
  (I-W)^2M_{k,m}\approx m\,\Dder(r+1,k).
\]
This relation is still not exact.  Its value was conceptual: after removing
a local homogeneous part, the remaining defect became much simpler.  This
suggested that the true recurrence should be found by choosing the correct
homogeneous operator and then studying its residual.

\subsection{Correcting the local structure}
The next step was to correct the coefficients of the two-neighbour
homogeneous part.  The more accurate local backbone was
\begin{equation}\label{eq:corrected-backbone-merged}
  M_{k,m}
  \approx
  (2k-m)M_{k-1,m-1}
  +(2k-m-1)M_{k-1,m}
  +\operatorname{source}(k,m).
\end{equation}
With a source written directly in defect-derangement terms, the best compact
candidate was
\begin{equation}\label{eq:corrected-source-merged}
\begin{aligned}
M_{k,m} \approx {}&
(2k-m)M_{k-1,m-1}
+(2k-m-1)M_{k-1,m}  \\
&+(k-m+1)\Dder(r,k)
+(k+m-1)\Dder(r+1,k)  \\
&-(k+m-1)\Dder(r,k-1)
+(km+1)\Dder(r+1,k-1).
\end{aligned}
\end{equation}
In the original $(k,r)$-coordinates this becomes
\begin{equation}\label{eq:corrected-source-r-merged}
\begin{aligned}
B_{k,2,r}\approx{}&
rB_{k-1,2,r-1}
+(r-1)B_{k-1,2,r-2}  \\
&+(r-k+1)\Dder(r,k)
+(3k-r-1)\Dder(r+1,k)  \\
&-(3k-r-1)\Dder(r,k-1)
+\bigl(k(2k-r)+1\bigr)\Dder(r+1,k-1).
\end{aligned}
\end{equation}
Again, \eqref{eq:corrected-source-r-merged} is not an exact formula.
However, it correctly identified the homogeneous part.  In particular, the
terms
\[
  rB_{k-1,2,r-1}
  +(r-1)B_{k-1,2,r-2}
\]
are precisely the binomial-basis coefficients obtained by multiplying the
previous row by a linear factor.  This was the main structural improvement.

\subsection{Passing to lower modes.}
The direct $\Dder$-description was too rigid.  The next change was to
work with the lower Schur modes themselves:
\[
  a_{n,r}:=B_{n,0,r},\qquad
  b_{n,r}:=B_{n,1,r},\qquad
  c_{n,r}:=B_{n,2,r}.
\]
This is natural because the positivity of the $a$- and $b$-rows had
already been proved in Theorem~\ref{thm:rank-two-known-A}.

For a sequence $f=(f_r)_{r\ge0}$, define
\[
  (T_\alpha f)_r:=(r-\alpha)f_r+r f_{r-1},
  \qquad f_{-1}:=0.
\]
This operator has a simple binomial-basis meaning.  If
\[
  F(x)=\sum_{r\ge0}f_r\binom{x}{r},
\]
then
\[
  (x-\alpha)F(x)
  =
  \sum_{r\ge0}(T_\alpha f)_r\binom{x}{r}.
\]
Thus the operators $T_\alpha$ encode multiplication by linear factors in
the binomial basis.

The most informative lower-mode candidate was
\begin{equation}\label{eq:operator-library-merged}
  c_{k,\bullet+1}
  \approx
  T_{-1}(c_{k-1,\bullet})
  +2T_{-1}(b_{k-1,\bullet})
  -T_{k-2}(b_{k-1,\bullet})
  +T_1(b_{k-2,\bullet}),
\end{equation}
where $c_{k,\bullet+1}$ denotes the shifted sequence
$r\mapsto c_{k,r+1}$.  Expanding the operators gives
\begin{equation}\label{eq:operator-library-expanded-merged}
\begin{aligned}
c_{k,r+1}\approx{}&
(r+1)c_{k-1,r}
+r c_{k-1,r-1}  \\
&+(r+k)b_{k-1,r}
+r b_{k-1,r-1}  \\
&+(r-1)b_{k-2,r}
+r b_{k-2,r-1}.
\end{aligned}
\end{equation}
This formula was not exact, but it was the decisive clue.  It suggested
that the correct recurrence should be written not in terms of raw
$\Dder$-atoms, but in terms of the lower binomial coefficient arrays
$a$, $b$, and $c$, acted on by the elementary operators $T_\alpha$.

There is also a natural reason for the shifted left-hand side
$c_{k,\bullet+1}$.  Since
\[
  \Delta\binom{x}{r+1}=\binom{x}{r},
\]
one has
\[
  \Delta A_{k,2}(x)
  =
  \sum_{r\ge0} c_{k,r+1}\binom{x}{r}.
\]
Thus the search was implicitly pointing toward a recurrence for the forward
difference of $A_{k,2}$, rather than directly for $A_{k,2}$ itself.

\subsection{Conclusion of the AlphaEvolve search.}
The mathematical lesson of the search can be summarized as follows: the positivity of $B_{k,2,r}$ should be proved by a binomial-basis recurrence using lower modes.
In hindsight, this is exactly what happens in the proof of Theorem~\ref{thm:A1-solved}.  The exact identity
\eqref{eq:A12-forward-difference} expresses $\Delta A_{k,2}$ as a sum of
terms involving
\[
  A_{k-1,2},\qquad
  A_{k-1,1},\qquad
  A_{k-2,1},\qquad
  A_{k-2,0},\qquad
  A_{k-3,0},
\]
multiplied by explicit linear or quadratic factors.  In binomial-basis
coordinates, those multiplications are controlled by the elementary
operator $T_\alpha$ and by the binomial shift lemma
(Lemma~\ref{lem:binomial-shift}).  The approximate formulas in this section
therefore did not constitute a proof, but they identified the correct
coordinate system in which the proof becomes elementary.

\section{Proofs of Conjectures \ref{conj:A} and \ref{conj:B} by Co-Mathematician}\label{sec:detective-proofs}
The proof drafts produced by Co-Mathematician transformed the structural information suggested by the AlphaEvolve experiments into rigorous arguments. 

For Conjecture~3.2, the proof is carried out directly in the binomial
basis.  A binomial-shift lemma and an explicit recurrence express the
coefficients in the case $j=2$ in terms of binomial-basis coefficients
from lower Schur modes (i.e coefficients with smaller second part $j$), whose positivity has already been established.

For Conjecture~3.3, the proof uses a more delicate argument.  The
coefficients for $j=1$ are first expressed in terms of numbers of
fixed-point-free permutations with a prescribed number of cycles.  Strong
log-concavity properties of these permutation numbers reduce the desired
inequalities to a finite family of explicit polynomial inequalities, which
are then checked exactly.
\begin{lemma}[Binomial shift lemma]\label{lem:binomial-shift}
Let
\[
  A(x)=\sum_{r\ge I} a_r\binom{x}{r}
\]
be a polynomial with $a_r\ge0$ for all $r$, and suppose that $a_I>0$.  If $I\ge c$, then
\[
  (x-c)A(x)=\sum_{r\ge \max(I,c+1)} b_r\binom{x}{r}
\]
with $b_r\ge0$ for all $r$.
\end{lemma}

\begin{proof}
Using the absorption identity
\[
  x\binom{x}{r}=r\binom{x}{r}+(r+1)\binom{x}{r+1},
\]
we obtain
\[
  (x-c)A(x)=a_I(I-c)\binom{x}{I}
  +\sum_{r\ge I+1}\bigl(a_r(r-c)+r\,a_{r-1}\bigr)\binom{x}{r}.
\]
Since $I\ge c$ and $a_I>0$, the leading coefficient $a_I(I-c)$ is nonnegative.  For $r\ge I+1$ one has $r-c\ge1$, so all displayed coefficients are nonnegative.  If $I=c$, then the leading coefficient vanishes and the support shifts up by one.  This proves the claim.
\end{proof}

\begin{theorem}[Conjecture \ref{conj:A} for $j=2$]\label{thm:A1-solved}
For all integers $k\ge4$ and $r\ge0$ one has
\[
  B_{k,2,r}\ge0.
\]
Equivalently, Conjecture~\ref{conj:A} holds for $j=2$.
\end{theorem}

\begin{proof}
Define the generating polynomial
\[
  E_d(T,u):=\prod_{i=0}^d \bigl(T+i+(d-i)u\bigr)
  =\sum_{k=0}^{d+1} C_{k,d}(u)T^{d+1-k}.
\]
Factoring out the $i=d$ term gives
\[
  E_d(T,u)=(T+d)\,E_{d-1}(T+u,u).
\]
Expanding $E_{d-1}(T+u,u)$ and matching the coefficient of $T^{d+1-k}$ yields
\[
  C_{k,d}(u)=
  \sum_{j=0}^{k}\binom{d-j}{k-j}u^{k-j}C_{j,d-1}(u)
  +d\sum_{j=0}^{k-1}\binom{d-j}{k-1-j}u^{k-1-j}C_{j,d-1}(u).
\]
Multiplying by $1-u$ and extracting the coefficient of $u^2$, we obtain
\begin{align}
A_{k,2}(d)
={}&A_{k,2}(d-1)+(d-k+1)A_{k-1,1}(d-1)+\binom{d-k+2}{2}A_{k-2,0}(d-1) \nonumber\\
&+d\,A_{k-1,2}(d-1)+d(d-k+2)A_{k-2,1}(d-1)
+d\binom{d-k+3}{2}A_{k-3,0}(d-1). \label{eq:A12-recurrence-d}
\end{align}
Now put $x=d-1$ and write $\Delta f(x)=f(x+1)-f(x)$.  Then \eqref{eq:A12-recurrence-d} becomes
\begin{align}
\Delta A_{k,2}(x)
={}&(x-k+2)A_{k-1,1}(x)+\binom{x-k+3}{2}A_{k-2,0}(x) \nonumber\\
&+(x+1)A_{k-1,2}(x)+(x+1)(x-k+3)A_{k-2,1}(x)
+(x+1)\binom{x-k+4}{2}A_{k-3,0}(x). \label{eq:A12-forward-difference}
\end{align}
Since
\[
  \Delta A_{k,2}(x)=\sum_{r\ge0} B_{k,2,r+1}\binom{x}{r},
\]
it suffices to prove that every summand on the right-hand side of \eqref{eq:A12-forward-difference} has nonnegative binomial coefficients.

By Theorem~\ref{thm:rank-two-known-A}, the polynomials $A_{n,0}(x)$ and $A_{n,1}(x)$ have nonnegative binomial coefficients.  Moreover their minimum supports satisfy
\[
  \operatorname{mindeg}_{\binom{\cdot}{r}}A_{n,0}\ge n,
  \qquad
  \operatorname{mindeg}_{\binom{\cdot}{r}}A_{n,1}\ge n-1.
\]
We prove by induction on $k$ that $A_{k,2}(x)$ has nonnegative binomial coefficients and minimum support at least $k-1$.

The base case $k=3$ is trivial, since $A_{3,2}(x)=0$.  Assume the claim for $A_{k-1,2}(x)$.  Then Lemma~\ref{lem:binomial-shift} applies successively to each term on the right-hand side of \eqref{eq:A12-forward-difference}:
\begin{itemize}
\item $(x-(k-2))A_{k-1,1}(x)$ has nonnegative coefficients and minimum support at least $k-1$;
\item $\binom{x-k+3}{2}A_{k-2,0}(x)=\frac12(x-(k-3))(x-(k-2))A_{k-2,0}(x)$ has nonnegative coefficients and minimum support at least $k-1$;
\item $(x+1)A_{k-1,2}(x)$ has nonnegative coefficients and minimum support at least $k-2$;
\item $(x+1)(x-k+3)A_{k-2,1}(x)$ has nonnegative coefficients and minimum support at least $k-2$;
\item $(x+1)\binom{x-k+4}{2}A_{k-3,0}(x)=\frac12(x+1)(x-k+4)(x-k+3)A_{k-3,0}(x)$ has nonnegative coefficients and minimum support at least $k-2$.
\end{itemize}
Hence $\Delta A_{k,2}(x)$ has nonnegative binomial coefficients and minimum support at least $k-2$, so
\[
  B_{k,2,r+1}\ge0\qquad(r\ge0).
\]
Finally,
\[
  B_{k,2,0}=A_{k,2}(0)=0,
\]
which is immediate from \eqref{eq:Akj-stirling}.  Therefore all coefficients $B_{k,2,r}$ are nonnegative, and $A_{k,2}(x)$ has minimum support at least $k-1$.  This closes the induction.
\end{proof}

The proof of Conjecture~\ref{conj:B} for $j=1$ starts from the explicit formula of Theorem~\ref{thm:rank-two-known-A}.  Set
\[
  D_i:=\Dder(r+i,k-1)\qquad(i\in\{-3,-2,-1,0,1,2\}),
\]
and write
\[
  B_{k,1,r}=p_1D_1+p_0D_0+p_{-1}D_{-1}+p_{-2}D_{-2},
\]
where
\begin{align*}
p_1&=r(r+2-k),\\
p_0&=r(3r+2-3k),\\
p_{-1}&=3r^2-3kr-2r+k+1,\\
p_{-2}&=(r-1)(r-k-1).
\end{align*}

\begin{lemma}[Strong log-concavity of $\Dder(\cdot,K)$]\label{lem:strong-log-D}
Fix $K\ge0$.  For all integers $p\le q$ one has
\[
  \Dder(p,K)\Dder(q,K)\ge \Dder(p-1,K)\Dder(q+1,K).
\]
\end{lemma}

\begin{proof}
In the proof of Theorem~\ref{thm:rank-two-known-B} we showed that the generating polynomial
\[
  P_K(t)=\sum_{N}\Dder(N,K)t^N
\]
has only real nonpositive roots.  Hence its coefficient sequence is log-concave and has no internal zeros.  Therefore the ratios
\[
  \frac{\Dder(N,K)}{\Dder(N-1,K)}
\]
form a weakly decreasing sequence on the active support.  Iterating this monotonicity from $p$ to $q$ gives
\[
  \frac{\Dder(p,K)}{\Dder(p-1,K)}
  \ge
  \frac{\Dder(q+1,K)}{\Dder(q,K)},
\]
which is equivalent to the displayed inequality.
\end{proof}

\begin{theorem}[Conjecture~\ref{conj:B} for $j=1$]\label{thm:B1-solved}
For all integers $k\ge2$ and $r\ge0$ one has
\[
  B_{k,1,r}^2\ge B_{k,1,r-1}B_{k,1,r+1}.
\]
Equivalently, Conjecture~\ref{conj:B} holds for $j=1$.
\end{theorem}

\begin{proof}
Set
\[
  L_{k,r}:=B_{k,1,r}^2-B_{k,1,r-1}B_{k,1,r+1}.
\]
Expanding $L_{k,r}$ in the variables $D_{-3},\ldots,D_2$ gives a quadratic form.  Group its terms by the trace $S=a+b$.  For each $S\in\{-4,-3,-2,-1,0,1,2\}$, let
\[
  (m_0,M_0),\ldots,(m_t,M_t)
\]
be the pairs with $m_i\le M_i$, $m_i+M_i=S$, ordered from the center outward:
\begin{align*}
S=-4:&\quad (-2,-2),(-3,-1),\\
S=-3:&\quad (-2,-1),(-3,0),\\
S=-2:&\quad (-1,-1),(-2,0),(-3,1),\\
S=-1:&\quad (-1,0),(-2,1),(-3,2),\\
S=0:&\quad (0,0),(-1,1),(-2,2),\\
S=1:&\quad (0,1),(-1,2),\\
S=2:&\quad (1,1),(0,2).
\end{align*}
Write
\[
  U_{S,i}:=D_{m_i}D_{M_i}.
\]
By Lemma~\ref{lem:strong-log-D}, the sequence $U_{S,0},U_{S,1},\ldots,U_{S,t}$ is weakly decreasing.  Define
\[
  \Delta_{S,i}:=U_{S,i}-U_{S,i+1}\ge0,
\qquad
  U_{S,t+1}:=0.
\]

Let $c_{S,i}(k,r)$ be the coefficient of $U_{S,i}$ in the slice of $L_{k,r}$ with trace $S$.  By discrete Abel summation,
\[
  \sum_{i=0}^{t} c_{S,i}U_{S,i}
  =
  \sum_{i=0}^{t} W_{S,i}\Delta_{S,i},
  \qquad
  W_{S,i}:=\sum_{j=0}^{i} c_{S,j}.
\]
Therefore
\[
  L_{k,r}=\sum_{S=-4}^{2}\sum_i W_{S,i}(k,r)\,\Delta_{S,i}.
\]
It is thus enough to prove that every prefix sum $W_{S,i}(k,r)$ is nonnegative on the active support of $\Delta_{S,i}$.

Fix one of the pairs $(m_i,M_i)$.  If $\Delta_{S,i}\neq0$, then both indices $r+m_i$ and $r+M_i$ lie in the active support of $\Dder(\,\cdot\,,k-1)$, namely
\[
  k\le r+m_i,\qquad r+M_i\le 2k-2.
\]
Introduce new variables
\[
  X:=r+m_i-k\ge0,\qquad Y:=2k-2-(r+M_i)\ge0.
\]
Solving for $k$ and $r$ gives
\[
  k=X+Y+M_i-m_i+2,\qquad
  r=2X+Y+M_i-2m_i+2.
\]
Substituting these expressions into each $W_{S,i}(k,r)$ yields a polynomial $W_{S,i}(X,Y)$ in the nonnegative variables $X$ and $Y$.  The exact list of all seventeen resulting polynomials is recorded in Appendix~\ref{app:prefix}.  Every one of them has nonnegative coefficients.  Hence
\[
  W_{S,i}(X,Y)\ge0
  \qquad
  \text{for all }X,Y\ge0.
\]
Since also $\Delta_{S,i}\ge0$, every term in the Abel decomposition of $L_{k,r}$ is nonnegative.  Therefore $L_{k,r}\ge0$ for all $k,r$, proving Conjecture~\ref{conj:B} for $j=1$.
\end{proof}

\begin{remark}
The proof of Theorem~\ref{thm:B1-solved} is exact but computer-assisted.  The role of the computer algebra calculation is finite and transparent: it expands the seventeen prefix-sum polynomials $W_{S,i}(X,Y)$ and checks that every coefficient is nonnegative.  No numerical approximation is involved.
\end{remark}

\section{Shifted binomial log-concavity for Pl\"ucker coefficients}
\label{sec:plucker-shifted-binomial-lc}

In this section we discuss a related binomial log-concavity phenomenon for the
Pl\"ucker coefficients of coincident root strata. This problem is a harder, stratified version of Theorem \ref{thm:A1-solved} and \ref{thm:B1-solved}.
Let
$\lambda=(2^{e_2},\ldots,m^{e_m})$ be a partition without parts equal
to $1$, and let
\[
  N:=|\lambda|,\qquad M:=|\tilde\lambda|,
\]
where $\tilde\lambda$ is obtained from $\lambda$ by subtracting $1$
from each part.  The equivariant cohomology class 
$[\overline Y_\lambda(d)]$ was introduced in \eqref{plucker} of the Introduction, and it is computed recursively by
Theorem~2.7 of~\cite{feher2025polynomiality}.  In the Schur basis it has
the form
\[
  [\overline Y_\lambda(d)]
  =
  \sum_{\lceil M/2\rceil\le i\le M}
  P_{\lambda;i}(d)\,s_{i,M-i},
  \qquad
  P_{\lambda;i}(d):=\operatorname{Pl}_{\lambda;2i-M}(d).
\]
Thus the admissible indices are precisely $\lceil M/2\rceil\le i\le M$.

The recursive formula is naturally valid at the enumerative threshold
$d\ge N$.  This suggests shifting the degree variable to
\[
  d=N+x.
\]
For each admissible pair $(\lambda,i)$, define the shifted Pl\"ucker
polynomial
\[
  Q_{\lambda;i}(x):=P_{\lambda;i}(N+x),
\]
and expand it in the binomial basis of the shifted variable:
\[
  Q_{\lambda;i}(x)=\sum_{r\ge0} C_{\lambda;i,r}\binom{x}{r}.
\]
Equivalently,
\[
  C_{\lambda;i,r}=\Delta_x^r Q_{\lambda;i}(0),
\]
where $\Delta_x f(x)=f(x+1)-f(x)$ is the forward difference operator.

The ordinary binomial basis in $d$ is not the natural one in this
problem.  For example,
\[
  \operatorname{Pl}_{(3);0}(d)
  =
  6\binom d2-3\binom d1
\]
already has a negative binomial coefficient.  After the shift
$d=N+x$, however, all computed examples exhibit positive and
log-concave binomial-basis coefficients.

\begin{conjecture}[Shifted-binomial log-concavity for Pl\"ucker coefficients]
\label{conj:shifted-plucker-binomial-lc}
For every partition $\lambda$ without parts equal to $1$, and for every
admissible index $i$, the coefficient sequence
\[
  (C_{\lambda;i,0},C_{\lambda;i,1},C_{\lambda;i,2},\ldots)
\]
is nonnegative and log-concave.  Equivalently, $C_{\lambda;i,r}\ge 0$ for all $r$, and
\[
  C_{\lambda;i,r}^2
  \ge
  C_{\lambda;i,r-1}C_{\lambda;i,r+1}
\]
for all relevant $r$, with the convention that $C_{\lambda;i,r}=0$
outside the degree range of $Q_{\lambda;i}$.
\end{conjecture}

By Theorem~\ref{thm:binomial-lc-implies-lc}, binomial positivity together
with binomial log-concavity implies ordinary log-concavity of the values.
Thus Conjecture~\ref{conj:shifted-plucker-binomial-lc} formally implies
the following weaker shifted-value conjecture.

\begin{conjecture}[Shifted Pl\"ucker value log-concavity]
\label{conj:shifted-plucker-value-lc}
For every admissible pair $(\lambda,i)$, the sequence
\[
  Q_{\lambda;i}(0),Q_{\lambda;i}(1),Q_{\lambda;i}(2),\ldots
\]
is log-concave after deleting its initial zeroes.  Equivalently,
\[
  P_{\lambda;i}(N+x)^2
  \ge
  P_{\lambda;i}(N+x-1)P_{\lambda;i}(N+x+1)
\]
for all $x\ge1$ in the nonzero range.
\end{conjecture}

\begin{remark}
 Conjecture~\ref{conj:shifted-plucker-value-lc} is stated separately
because it is the directly visible numerical consequence at the level of
Pl\"ucker values.  The stronger structure is
Conjecture~\ref{conj:shifted-plucker-binomial-lc}: it predicts
log-concavity before the binomial transform is evaluated.
\end{remark}

\begin{example}[Top coefficient]
\label{ex:plucker-top-shifted-binomial}
The top Pl\"ucker coefficient is known explicitly:
\[
  \operatorname{Pl}_{\lambda;M}(d)
  =
  \frac{1}{\prod_{q\ge2} e_q!}\,
  d(d-1)\cdots(d-N+1).
\]
After shifting $d=N+x$, this becomes
\[
  \operatorname{Pl}_{\lambda;M}(N+x)
  =
  \frac{1}{\prod_{q\ge2} e_q!}\,
  (x+1)(x+2)\cdots(x+N).
\]
Therefore
\[
  \operatorname{Pl}_{\lambda;M}(N+x)
  =
  \frac{N!}{\prod_{q\ge2} e_q!}
  \sum_{r=0}^{N}\binom Nr\binom xr.
\]
Thus the shifted-binomial coefficient sequence is a positive multiple of
\[
  \binom N0,\binom N1,\ldots,\binom NN,
\]
and is therefore log-concave.

For instance, for $\lambda=(2,2)$, we have $N=4$, $M=2$, and
\[
  \operatorname{Pl}_{(2,2);2}(d)
  =
  \frac12 d(d-1)(d-2)(d-3).
\]
Hence
\[
  \operatorname{Pl}_{(2,2);2}(4+x)
  =
  12
  +48\binom x1
  +72\binom x2
  +48\binom x3
  +12\binom x4,
\]
with coefficient sequence
\[
  (12,48,72,48,12).
\]
\end{example}

\begin{example}[The bitangent coefficient]
For $\lambda=(2,2)$, the ordinary bitangent polynomial is
\[
  \operatorname{Pl}_{(2,2);0}(d)
  =
  \frac12 d(d-2)(d-3)(d+3).
\]
After shifting by $N=4$,
\[
  \operatorname{Pl}_{(2,2);0}(4+x)
  =
  \frac12(x+4)(x+2)(x+1)(x+7).
\]
Its shifted-binomial expansion is
\[
  \operatorname{Pl}_{(2,2);0}(4+x)
  =
  28
  +92\binom x1
  +112\binom x2
  +60\binom x3
  +12\binom x4.
\]
Thus
\[
  (C_0,C_1,C_2,C_3,C_4)
  =
  (28,92,112,60,12),
\]
and the log-concavity inequalities are
\[
  92^2\ge 28\cdot112,\qquad
  112^2\ge 92\cdot60,\qquad
  60^2\ge 112\cdot12.
\]
\end{example}

\begin{example}[A non-top coefficient for $\lambda=(2,2,2)$]
For $\lambda=(2,2,2)$, one has $N=6$, $M=3$, and
\[
  \operatorname{Pl}_{(2,2,2);1}(d)
  =
  \frac13 d(d-5)(d-4)(d-3)(d^2+3d-2).
\]
After shifting,
\[
  \operatorname{Pl}_{(2,2,2);1}(6+x)
  =
  \frac13(x+6)(x+1)(x+2)(x+3)(x^2+15x+52).
\]
The shifted-binomial expansion is
\[
\begin{aligned}
  \operatorname{Pl}_{(2,2,2);1}(6+x)
  ={}&
  624
  +3184\binom x1
  +6768\binom x2
  +7680\binom x3  \\
  &+4912\binom x4
  +1680\binom x5
  +240\binom x6.
\end{aligned}
\]
Thus
\[
  (C_0,C_1,C_2,C_3,C_4,C_5,C_6)
  =
  (624,3184,6768,7680,4912,1680,240),
\]
which is positive and log-concave.
\end{example}

\subsection{A proof strategy and the current gap}

The following is a revised version of the induction strategy suggested by
the Co-Mathematician proof draft.  It should be read as a proof idea, not
as a complete proof.

Let $m$ be the largest part of $\lambda$, let $\lambda'$ be obtained
by removing one part $m$, and put
\[
  N'=|\lambda'|=N-m.
\]
The recursion of Theorem~2.7 of~\cite{feher2023plucker} is compatible
with the shift $d=N+x$, because
\[
  d-m=N'+x.
\]
Thus the recursion takes the shifted form
\[
  [\overline Y_\lambda(N+x)]
  =
  \frac1{e_m}
  \partial
  \left(
    [\overline Y_{\lambda'}(N'+x)]_{m/(N'+x)}
    \prod_{\ell=0}^{m-1}\bigl(\ell a+(N+x-\ell)b\bigr)
  \right),
\]
where
\[
  \partial(\alpha)(a,b)
  =
  \frac{\alpha(a,b)-\alpha(b,a)}{b-a}.
\]
This gives a natural induction on the length of $\lambda$.  Since
\[
  C_{\lambda;i,r}
  =
  \Delta_x^r Q_{\lambda;i}(0),
\]
the conjecture can be reformulated as a statement about the finite
differences in $x$ of the Schur coefficients produced by this recursion.

A possible route to a proof would be to turn the shifted recursion into a
coefficient-level transition formula of the form
\[
  C_{\lambda;i,r}
  =
  \sum_{i',s}
  K_{\lambda,m}^{i,i'}(r,s)\,
  C_{\lambda';i',s},
\]
where the kernels $K_{\lambda,m}^{i,i'}(r,s)$ are manifestly nonnegative.
If these kernels were known to preserve log-concavity in the $r$-direction,
and if the sums over $i'$ satisfied the necessary compatibility or
interlacing conditions, then the conjecture would follow by induction from
the empty partition.

The main missing point is precisely this kernel statement.  Positivity of
the individual factors
\[
  \prod_{\ell=0}^{m-1}\bigl(\ell a+(N+x-\ell)b\bigr)
\]
after shifting is not enough: the divided difference $\partial$ mixes the
Schur components, and sums of log-concave sequences need not be log-concave
without an additional synchronization condition.  A complete proof would
therefore need a cancellation-free binomial-basis form of the shifted
recursion, together with a total-positivity or interlacing statement strong
enough to imply log-concavity after summing over the intermediate Schur
indices.

There are also two tempting approaches which do not by themselves complete
the proof.  First, one cannot simply prove real-rootedness and appeal to
Newton's inequalities.  For example, the coefficient of $s_{3,1}$ for
$\lambda=(3,3)$ gives
\[
  Q(x)
  =
  \frac12(x+1)(x+2)(x+6)
  \bigl(x^3+27x^2+188x+402\bigr),
\]
and the cubic factor has discriminant
\[
  -96548<0.
\]
Thus $Q(x)$ is not real-rooted.  Moreover, its shifted-binomial
coefficient sequence is
\[
  (2412,10566,19368,19026,10512,3060,360),
\]
whose ordinary generating polynomial factors as
\[
  18(z+1)^2
  \bigl(20z^4+130z^3+304z^2+319z+134\bigr),
\]
and the quartic factor has discriminant
\[
  -975021840<0.
\]
Thus the binomial-coefficient generating polynomial is not generally
real-rooted either.

Second, the divided difference $\partial$ in Theorem~2.7 acts on the
Chern-root variables $a,b$, not on the shifted degree variable $x$.
Therefore $\partial$ should not be confused with the forward difference
$\Delta_x$.  Similarly, the coefficients
\[
  C_{\lambda;i,r}=\Delta_x^r P_{\lambda;i}(N)
\]
are finite differences in an external degree parameter.  At present there
is no direct realization of these numbers as intersection numbers of powers
of a fixed nef divisor, so standard Hodge-theoretic inequalities such as
Khovanskii--Teissier or Lorentzian-polynomial log-concavity do not apply
without further geometric input.

Thus the recursion gives a promising induction scheme, but the essential
new ingredient would be a positive, log-concavity-preserving transition
kernel for the shifted-binomial coefficients.

\subsection{Chern--Schwartz--MacPherson classes of the coincident root strata}\label{subsec:csm-root-strata}

A natural generalization is to study the equivariant
Chern--Schwartz--MacPherson classes of the coincident root strata; see
\cite{juhasz-thesis}.  These classes lead to formulas for Euler
characteristics of varieties of lines with prescribed contact type
$\lambda$ to a degree $d$ hypersurface.  In the case of the empty
partition, this framework recovers the total equivariant Chern class of
$\Pol^d(\mathbb C^2)$.  Thus Pl\"ucker numbers and the total Chern class
$c(\Pol^d(\mathbb C^2))$ are closely related special cases of the same
general theory of equivariant characteristic classes of coincident root
strata.  In both settings one obtains polynomial dependence on $d$, and
the examples suggest stronger positivity and log-concavity after passing to
the appropriate binomial basis.

ChatGPT~5.5 Pro found the following conjecture:

\begin{conjecture}[Shifted binomial log-concavity for stable CSM--Schur coefficients]
Let $\lambda$ be a partition with parts at least $2$, and let
\[
c=\operatorname{codim}Y_\lambda=|\lambda|-\ell(\lambda).
\]
For every $f\ge c$, consider the homogeneous degree $f$ part of the stable
Chern--Schwartz--MacPherson class of the open coincident root stratum
$Y_\lambda(d)$:
\[
c^{SM}(Y_\lambda(d))_f.
\]
Expand it in the two-variable Schur basis as
\[
c^{SM}(Y_\lambda(d))_f
=
\sum_{i=0}^{\lfloor f/2\rfloor}
C_{\lambda,f,i}(d)\,s_{f-i,i}(a,b),
\]
where each $C_{\lambda,f,i}(d)$ is a polynomial in $d$ in the stable range.

Define the signed Schur coefficient polynomial
\[
\widetilde C_{\lambda,f,i}(d)
=
(-1)^{f-c}C_{\lambda,f,i}(d).
\]
Set
\[
y=d-|\lambda|-2(f-c).
\]
Then the shifted polynomial
\[
\widetilde C_{\lambda,f,i}\bigl(y+|\lambda|+2(f-c)\bigr)
\]
has a nonnegative expansion in the binomial basis:
\[
\widetilde C_{\lambda,f,i}\bigl(y+|\lambda|+2(f-c)\bigr)
=
\sum_{j\ge 0}
\gamma_{\lambda,f,i,j}\binom{y}{j},
\qquad
\gamma_{\lambda,f,i,j}\ge 0.
\]
Moreover, after deleting possible initial zeros, the coefficient sequence
\[
\left(\gamma_{\lambda,f,i,0},
\gamma_{\lambda,f,i,1},
\gamma_{\lambda,f,i,2},
\ldots
\right)
\]
is log-concave:
\[
\gamma_{\lambda,f,i,j}^{\,2}
\ge
\gamma_{\lambda,f,i,j-1}\gamma_{\lambda,f,i,j+1}
\qquad
\text{for all } j.
\]
\end{conjecture}

This conjecture should be viewed as a candidate common extension of
Conjecture~\ref{conj:B} and Conjecture~\ref{conj:shifted-plucker-binomial-lc}.
There is, however, an important distinction: the full vector space is naturally a
closure rather than an open stratum.  At present we do not know the correct
log-concavity statement for CSM or ordinary classes of stratum closures.
\appendix
\section{Prefix-sum polynomials in the proof of Theorem~\texorpdfstring{\ref{thm:B1-solved}}{B1}}\label{app:prefix}

For each trace slice $S$ in the proof of Theorem~\ref{thm:B1-solved}, the Abel prefix sums $W_{S,i}(k,r)$ become explicit polynomials in the support variables
\[
  X=r+m_i-k,\qquad Y=2k-2-(r+M_i),
\]
where $(m_i,M_i)$ is the $i$-th pair on the slice.  The following list was produced by exact symbolic expansion.  Every coefficient is nonnegative.

For $S=-4$:
\begin{align*}
W_{-4,0}(X,Y)={}&4X^4+4X^3Y+20X^3+X^2Y^2+14X^2Y+37X^2\\
&+2XY^2+16XY+30X+Y^2+6Y+9,\\
W_{-4,1}(X,Y)={}&5X^2+4XY+28X+Y^2+12Y+39.
\end{align*}

For $S=-3$:
\begin{align*}
W_{-3,0}(X,Y)={}&12X^4+12X^3Y+84X^3+3X^2Y^2+58X^2Y+221X^2\\
&+8XY^2+94XY+258X+6Y^2+52Y+112,\\
W_{-3,1}(X,Y)={}&30X^2+24XY+204X+6Y^2+88Y+344.
\end{align*}

For $S=-2$:
\begin{align*}
W_{-2,0}(X,Y)={}&24X^4+24X^3Y+94X^3+6X^2Y^2+61X^2Y+140X^2\\
&+7XY^2+51XY+94X+2Y^2+14Y+24,\\
W_{-2,1}(X,Y)={}&12X^4+12X^3Y+110X^3+3X^2Y^2+77X^2Y+401X^2\\
&+11XY^2+181XY+696X+15Y^2+168Y+481,\\
W_{-2,2}(X,Y)={}&75X^2+60XY+598X+15Y^2+258Y+1187.
\end{align*}

For $S=-1$:
\begin{align*}
W_{-1,0}(X,Y)={}&32X^4+32X^3Y+196X^3+8X^2Y^2+134X^2Y+442X^2\\
&+18XY^2+174XY+432X+8Y^2+68Y+152,\\
W_{-1,1}(X,Y)={}&4X^4+4X^3Y+46X^3+X^2Y^2+33X^2Y+270X^2\\
&+5XY^2+145XY+832X+20Y^2+272Y+976,\\
W_{-1,2}(X,Y)={}&100X^2+80XY+912X+20Y^2+392Y+2080.
\end{align*}

For $S=0$:
\begin{align*}
W_{0,0}(X,Y)={}&24X^4+24X^3Y+78X^3+6X^2Y^2+53X^2Y+94X^2\\
&+7XY^2+37XY+50X+2Y^2+8Y+10,\\
W_{0,1}(X,Y)={}&12X^4+12X^3Y+102X^3+3X^2Y^2+73X^2Y+339X^2\\
&+11XY^2+151XY+522X+11Y^2+112Y+309,\\
W_{0,2}(X,Y)={}&75X^2+60XY+558X+15Y^2+238Y+1027.
\end{align*}

For $S=1$:
\begin{align*}
W_{1,0}(X,Y)={}&12X^4+12X^3Y+68X^3+3X^2Y^2+50X^2Y+138X^2\\
&+8XY^2+60XY+116X+3Y^2+18Y+32,\\
W_{1,1}(X,Y)={}&30X^2+24XY+172X+6Y^2+72Y+240.
\end{align*}

For $S=2$:
\begin{align*}
W_{2,0}(X,Y)={}&4X^4+4X^3Y+12X^3+X^2Y^2+10X^2Y+13X^2\\
&+2XY^2+8XY+6X+Y^2+2Y+1,\\
W_{2,1}(X,Y)={}&5X^2+4XY+20X+Y^2+8Y+19.
\end{align*}

\bibliographystyle{amsalpha}
\bibliography{refs}

\end{document}